\documentclass[11pt,a4paper]{amsart}

\usepackage{amsmath,amsthm,amsfonts,amssymb}
\usepackage[a4paper,left=25mm,right=25mm,top=30mm,bottom=30mm,marginpar=25mm]{geometry}
\usepackage{upref}

\def\R{\mathbb{R}}

\def\N{\mathbb{N}}

\def\Tr{\mathop{\rm Tr}}
\def\Id{{\rm Id}}

\def\Lip{\mathop{\rm Lip}}
\def\div{\mathop{\rm div}}

\def\dist{{\rm dist}}

\usepackage{graphicx}
\usepackage[latin1]{inputenc}
\numberwithin{equation}{section}

\newtheorem{theorem}{Theorem}[section]
\newtheorem{lemma}[theorem]{Lemma}

\newtheorem{corollary}[theorem]{Corollary}

\def\tilde{\widetilde}

\newcommand{\MM}[2]{\mathbb{R}^{#1\times #2}}

\newcommand{\dv}[1]{\,{\rm d}#1}

\newcommand{\bref}[1]{(\ref{#1})}

\def\epsilon{\varepsilon}

\def\loc{_{\mathrm{loc}}}

  \newcounter{constantsnumber}
  \def\setc#1{\refstepcounter{constantsnumber}%
  \label{const#1}c_{\theconstantsnumber}}
  \def\const#1{c_{\ref{const#1}}}

\def\SO{{\rm SO}}
\def\Ftilde{\widetilde F}    
\def\Gtilde{\widetilde G}
\def\fext{\widetilde f}    
\def\gext{\widetilde g}
\def\Fext{\widetilde F}    
\def\Gext{\widetilde G}
\def\uext{\widetilde u}
\def\uf{u_f}    
\def\ug{u_g}
\def\uftilde{\tilde u_f}    
\def\ugtilde{\tilde u_g}

\title[Geometric rigidity with mixed growth conditions]{Korn's second inequality and geometric rigidity with mixed growth conditions}

\author{Sergio Conti}
\address{Institut f\"ur Angewandte Mathematik, Universit\"at Bonn, 
Endenicher Allee 60, 53115 Bonn, Germany.}
\email{sergio.conti@uni-bonn.de}

\author{Georg Dolzmann}
\address{Fakult\"at f\"ur Mathematik, Universit\"at Regensburg, 93040 Regensburg, Germany.}
\email{georg.dolzmann@mathematik.uni-r.de}

\author{Stefan M\"uller}
\address{Institut f\"ur Angewandte Mathematik, Universit\"at Bonn, 
Endenicher Allee 60, 53115 Bonn, Germany.}
\email{stefan.mueller@hcm.uni-bonn.de}

\begin{document}

\begin{abstract}
Geometric rigidity states that a gradient field which is $L^p$-close to the
set of proper rotations is necessarily $L^p$-close to a fixed rotation, and is
one key estimate in nonlinear elasticity. In several applications, as for
example in the theory of plasticity,  energy densities with mixed growth
appear. We show here that  geometric rigidity holds also in $L^p+L^q$ and in
 $L^{p,q}$ interpolation spaces. As a first step we prove the corresponding linear
inequality, which generalizes Korn's inequality to these spaces.
\end{abstract}


\thanks{This work was partially supported by the Deutsche Forschungsgemeinschaft
through the Forschergruppe 797  {\em ``Analysis and computation of
  microstructure in finite plasticity''}, projects CO 304/4-2 (first author),
DO 633/2-1 (second author), and MU 1067/9-2 (third author).}

\date{\today}

\maketitle

\section{Introduction}

Since Korn's original contributions~\cite{Korn1906,Korn1908,Korn1909}, Korn's inequality has
played a central role in the analysis of boundary value problems in linear elasticity. 
In its basic form, Korn's inequality asserts the following. Suppose that
$\Omega\subset\R^n$ 
is a bounded, connected, Lipschitz domain and that $u\in
H^1(\Omega;\R^n)$. Then there exists a
skew-symmetric matrix $S$ such that $\|Du-S\|_2\leq c \|Eu\|_2$.  
That is, the $L^2$-norm of the skew-symmetric part of $Du$ is dominated
by the $L^2$-norm  of the symmetric part, after a suitable constant
$S$ has been subtracted.
Numerous generalizations to different boundary conditions, growth conditions
and unbounded domains have been given in the literature, see, 
e.g., \cite{Friedrichs1947,HavlacekNecasARMA1970I,KondratevOleinik1988} and the references therein.

In view of the fundamental importance of Korn's inequality in linear elasticity, it is 
not surprising that a suitable nonlinear version, which is often referred to as 
geometric rigidity, plays a central role in models in nonlinear elasticity. 
In their basic form, these estimates assert that for a deformation $u\in
H^1(\Omega; \R^n)$ 
the distance of $D u$ to a suitably chosen proper rotation $Q\in\SO(n)$ is
dominated in $L^2$ by the distance function of $Du$ to $\SO(n)$. The  
proof~\cite{FrieseckeJamesMueller2005} is based on the fact that the nonlinear estimate can be
related to the linear one since the 
tangent space to the smooth manifold $\SO(n)$ at the identity matrix is given by the
linear space of all skew-symmetric matrices. 

In fact,  geometric rigidity results are the cornerstone of 
rigorous derivations of two-dimensional
plate and shell theories from three-dimensional models in the framework of nonlinear elasticity
theory. The quantitative version by Friesecke, James and M\"uller~\cite{FrieseckeJamesMueller2005}
generalized previous work~\cite{JohnCPAM1961,John1972,Reshetnyak1967,Kohn82}
and allowed for 
the first time the derivation of limiting theories as the thickness of the three-dimensional
structure tends to zero without a~priori assumptions on the 
deformations in various scaling 
regimes~\cite{FrieseckeJamesMueller2005,FrieseckeJamesMueller2008}.

More recently, the analysis of variational models for the elastic and plastic behavior of
single crystals has led to the question of whether analogous estimates can be established under
mixed growth conditions. In this paper we generalize both Korn's 
inequality and  the corresponding nonlinear estimate
to this setting. Our main result is the following.
\begin{theorem}\label{theo:rigiditypqnonlinear}
  Let $\Omega\subset\R^n$ be a bounded and connected domain with Lipschitz boundary. Suppose
that  $1<p<q<\infty$ and that $u\in W^{1,1}(\Omega;\R^n)$, $f\in L^p(\Omega)$ and
$g\in L^q(\Omega )$ are given with
\begin{align}\label{estimatedist}
   \dist(D u, \SO(n)) = f + g \qquad a.e.\text{ in }\Omega\,.
\end{align}
Then there exist a constant $c$, matrix fields $F\in L^p(\Omega;\R^{n\times n})$,
$G\in L^q(\Omega;\R^{n\times n})$, and a proper rotation $Q\in \SO(n)$ such that
\begin{equation*}
  D u = Q + F + G \qquad a.e.\text{ in }\Omega\,,
\end{equation*}
and
\begin{align}\label{thmestimates}
  \|F\|_{L^p(\Omega;\R^{n\times n})}\le c \|f\|_{L^p(\Omega)}\,, \hskip1cm
  \|G\|_{L^q(\Omega;\R^{n\times n})}\le c \|g\|_{L^q(\Omega)}\,.
 \end{align}
The constant $c$ depends only on $n$, $p$, $q$, and $\Omega$ but not on  $u$, $f$, $g$.
\end{theorem}

The case $p=2$ and $g=0$ was established
in~\cite[Th. 3.1]{FrieseckeJamesMueller2005},  
the generalization to $p\in(1,\infty)$ follows from the same proof with minor
changes, see \cite[Sect. 2.4]{ContiSchweizer2006b}. This version 
with mixed growth conditions was first stated without proof
in~\cite[Prop. 5]{FrieseckeJamesMueller2008} and has already been used in 
\cite{AgostinianiDalmasoDesimone2012} to study nonlinear models with weak
coerciveness and in \cite{CaterinaLucia} to study 
models of geometrically necessary dislocations in finite elastoplasticity.    
Our result implies a statement on equiintegrability (see Corollary
\ref{corollaryequi}), which has been used in
\cite{FrieseckeJamesMueller2008} 
to show strong convergence of minimizing sequences.
We believe that the generalization of Korn's inequality which is the basis for
the 
proof presented here and is stated in Theorem 
\ref{theokorn} below, is of independent interest. In Section \ref{weakspaces}
we briefly discuss how the present results imply
estimates in Lorentz $L^{p,q}$ spaces, present the statement on
equiintegrability of sequences and generalize to more than two exponents.
 
\bigskip

\textit{Notation.} We use standard notation for Lebesgue and Sobolev spaces and
omit in the notation of their norms the domain and the range if they are clear
from the context.  We use $|E|$ for the Lebesgue measure of a measurable set $E\subset\R^n$.
For $u\colon\Omega\to\R^n$ we define the symmetric part
of the deformation gradient as $Eu=(Du+Du^T)/2$. We denote the trace
of a matrix $A\in \MM{n}{n}$ by $\Tr A$ and the inner product between to vectors
$a$, $b\in \R^n$ and two matrices $A$, $B\in\MM{}{n}$ by $a\cdot b$ and
$A:B=\Tr A^TB$, respectively.
The distance $\dist(\cdot, \SO(n))$ is the usual Euclidean
distance. We use the convention that constants may change from line to line
as long as they depend only on $n$, $p$, $q$ and $\Omega$.
Finally we use the fact that an estimate of the norm of a matrix field implies a
decomposition of the matrix field with estimates. More precisely, if $A\colon\Omega
\to \MM{n}{n}$ satisfies $|A|\leq f+g$ with $f\in L^p(\Omega)$, $g\in
L^q(\Omega)$ and $f,g\ge 0$,
then 
\begin{align}\label{getdecomposition}
 A = \frac{f}{f+g}\,\chi_{\{f+g\neq 0\}}\,A +  \frac{g}{f+g}\,\chi_{\{f+g\neq 0\}}\,A = F+G
\end{align}
with $\|F\|_p\leq \|f\|_p$ and $\|G\|_q\leq \|g\|_q$.
If $f$ and $g$ are not nonnegative, we replace them first by their absolute values.

\section{Linear estimate: Korn's inequality}

We start by the generalization of Korn's inequality to the case of mixed
growth. This result will also be the key ingredient into the proof 
of Theorem~\ref{theo:rigiditypqnonlinear}.
\begin{theorem}\label{theokorn}
  Let $\Omega\subset\R^n$ be a bounded and connected domain with  Lipschitz boundary.
Suppose that $1<p<q<\infty$ and that $u\in W^{1,1}(\Omega;\R^n)$, $f \in
  L^p(\Omega;\R^{n\times n})$, 
$g \in L^q(\Omega;\R^{n\times n})$ are given with
\begin{equation}\label{Eudecomp}
Eu=\frac{1}{2}\bigl(  D u + D u^T \bigr) 
=   f   + g  \qquad a.e.\text{ in }\Omega\,.
\end{equation}
Then there exists a constant $c$, matrix fields $F\in L^p(\Omega;\R^{n\times n})$,
$G\in L^q(\Omega;\R^{n\times n})$, and a skew-symmetric matrix 
$S\in \R^{n\times n}$, that is, $S+S^T=0$, such that
\begin{equation}\label{korndecomposition}
  D u = S + F + G\qquad a.e.\text{ in }\Omega\,,
\end{equation}
and
\begin{equation}\label{kornestimate}
  \|F\|_{L^p(\Omega)}\le c \| f  \|_{L^p(\Omega)}\,, \hskip1cm
  \|G\|_{L^q(\Omega)}\le c \|g \|_{L^q(\Omega)}\,.
\end{equation}
The constant depends only on $n$, $p$, $q$, and $\Omega$. 
\end{theorem}

\begin{proof}
Korn's second inequality in $L^p$ states that 
for every bounded connected Lipschitz set $\Omega$ and every $p\in(1,\infty)$ there is 
constant $c(\Omega, p)$ such that 
for every $u\in W^{1,p}(\Omega;\R^n)$ 
there exists  a  skew-symmetric matrix $S\in\MM{n}{n}$ with
\begin{align}\label{eqkornlp}
\| Du-S\|_p\leq c(\Omega, p) \| Eu\|_p\,, 
\end{align}
see, e.g., \cite{Ting1971}, \cite{Nitsche1981}
or \cite[Theorem 8]{KondratevOleinik1988} for a proof. 

From this we can easily prove the assertion  in the case
 $ \|g\|_{L^q(\Omega)}\le \| f \|_{L^p(\Omega)}$.
Indeed, H\"older's inequality implies 
\begin{align*}
  \|Eu\|_{L^p(\Omega)} \le \| f \|_{L^p(\Omega)}+ \|g\|_{L^p(\Omega)} \le  
 \| f \|_{L^p(\Omega)}+ c \|g\|_{L^q(\Omega)} \le (c+1)
 \| f \|_{L^p(\Omega)}\,,
\end{align*}
and the assertion holds with $F = Du-S$ and
$G=0$. We may thus assume that
\begin{align}\label{order}
\| f \|_{L^p(\Omega)}\le \|g\|_{L^q(\Omega)}\,.
\end{align}
The proof relies on a covering argument together with a local estimate which is
based on splitting $u$ into a harmonic part and a remainder.

\medskip

\noindent
{\em Step 1: A representation of $\Delta u$}. We begin with an expression for $\Delta u$ 
in terms of $Eu$ which is frequently used in proofs of Korn's inequality, see, 
e.g.,~\cite{KondratevOleinik1988}. Suppose that $U\subset\R^n$ is open and that
$u\in W^{1,1}\loc(U;\R^n)$. 
Let $\phi\in C_c^\infty(U;\R^n)$.
Multiplication of the identity $Du:D\phi + Du:(D\phi)^T=2Eu:D\phi$
with $u$ and integration yields
 \begin{align*}
 - \int_{U} u \Delta \phi\dv{y}
=\int_{U} D u : D \phi\dv{y}
 = \int_{U}\bigl( 2 Eu : D\phi- Du:(D\phi)^T \bigr)\dv{y}\,.
\end{align*}
Partial integration transforms
the last term into $\int_U u\cdot \div ((D\phi)^T) \dv{y}=\int_U
u\cdot D\div \phi\dv{y}=-\int_U  \div u \div \phi\dv{y}$, and since $\div u=\Tr
Du=\Tr Eu$ we conclude that
 \begin{align}\label{ueqn}
 - \int_{U} u \Delta \phi\dv{y}
 = \int_{U}\bigl( 2 Eu : D\phi- (\Tr Eu) \div \phi \bigr)\dv{y}\,,
\end{align}
that is, $\Delta u = 2\div Eu-D(\Tr Eu)$ in $U$ the sense of distributions.

\medskip

\noindent
{\em Step~2: Construction of a finite cover of $\overline\Omega$.} 
We choose for every $x\in\Omega$ an $r_x>0$ such that $B(x,r_x)\subset\Omega$. 
For every $x\in \partial\Omega$ we fix an $r_x>0$ with the following properties.
There exist orthonormal vectors
$v_1,\ldots,v_{n}\in\R^n$ which  determine
a coordinate system in $\R^{n}$ and a Lipschitz function $\phi_x\colon 
 \R^{n-1}\to \R$ such that $\phi_x(0)=0$ and
\begin{align*}
B(x,r_x) \cap  \partial\Omega
&\, = \bigl\{x+ \sum_{i=1}^{n}\xi_i v_i\colon
\xi\in B(0,r_x), \xi_n=\phi_x(\xi_1,\ldots, \xi_{n-1})
\bigr\} \,, \\
B(x,r_x) \cap  \Omega
&\, = \bigl\{x+ \sum_{i=1}^{n}\xi_i v_i\colon
\xi\in B(0,r_x), \xi_n<\phi_x(\xi_1,\ldots, \xi_{n-1})
\bigr\} \,,
\end{align*}
that is, the boundary is a Lipschitz graph and the domain lies on one side of the graph.
Such a choice of a coordinate system and a Lipschitz function 
is possible since $\Omega$ has a Lipschitz boundary. 

We denote by $L$ a uniform Lipschitz constant for all the functions $\phi_x$,
$x\in\partial\Omega$ (this exists since $\partial\Omega$ is compact). 
Let $\gamma = 1/(2\sqrt{1+L^2})$. 
By construction, $\{B(x,\gamma r_x/2)\}_{x\in\overline\Omega}$ is an open cover of $\overline \Omega$.
Since $\overline \Omega$ is compact we may choose
a finite subcover $(B(x_\ell, \gamma r_\ell/2))_{\ell=0,\ldots, M}$ of $\overline\Omega$. Moreover,
the finitely many balls $B_\ell =  B(x_\ell, \gamma r_\ell/2)$ satisfy
\begin{align}\label{eqdefalpha}
 \alpha=\min\bigl\{|B_i\cap B_j\cap\Omega|: i,\,j\in\{0,\ldots, M\}\,, 
B_i\cap B_j\cap\Omega\ne\emptyset\bigr\} >0\,.
\end{align}
All constants are allowed to depend on the smallest radius of the balls in the covering.

\medskip

\noindent
{\em Step 3: Interior estimate.}
Let $N=N(z)$ denote the fundamental solution for the Laplace operator $-\Delta$. For any
$\psi\in C_c^\infty(\R^n)$ and $i$, $j\in \{1,\ldots, n\}$ the function $D_i N \ast\psi =
N\ast D_i\psi$ satisfies  $N\ast D_i\psi\in L^1\loc(\R^n)$ and the partial derivative
with respect to $x_j$ can be represented by
\begin{align*}
 D_j\bigl( N\ast D_i\psi\bigr)(x) 
= \bigl( T_{ij} \psi \bigr)(x) - \frac{\delta_{ij}}{n}\,\psi(x)
\end{align*}
where
\begin{align*}
 \bigl( T_{ij} \psi \bigr)(x)
=\lim_{\epsilon\to 0}\Bigl[ 
 \int_{\R^n\setminus B(0,\epsilon)}D_iD_j N(y)  \psi(x-y)\dv{y}\Bigr]\,.
\end{align*}
Classical results on singular integrals~\cite[Theorem~2 in Section~3.2]{Stein1970}
ensure that the limit on the right-hand side exists in $L^p$ for all $p\in
(1,\infty)$ and that the operator
$T_{ij}$ can be extended to a bounded operator from $L^p(\R^n)$ to
itself. In particular  there exists a constant $A_p$ which depends
only on $p$ such that
\begin{align*}
 \| T_{ij}f\|_p \leq A_p \| f\|_p\quad\text{ for all }f\in L^p(\R^n)\,.
\end{align*}
Analogously we define for a vector field $\psi\in C_c^\infty(\R^n;\R^n)$ the
function  
$u_\psi=\sum_{i=1}^n D_iN\ast \psi^{(i)}$ which is a weak solution of the equation
$-\Delta u=\div \psi$ 
in $\R^n$. Again, this definition can be extended by approximation to vector
fields 
$f\in L^p(\R^n;\R^n)$ and one obtains the corresponding function 
$u_f=\sum_{i=1}^n D_iN\ast f^{(i)}$ which satisfies $u_f\in L^1\loc(\R^n)$, $\|D u_f\|_p
\leq A_p \| f\|_p$ and which is a weak solution of $-\Delta u = \div f$ in the sense that
\begin{align}\label{veryweak}
 \int_{\R^n} u_f \Delta \phi \dv{x}
=\int_{\R^n} f\cdot D\phi\dv{x}\qquad \text{ for all }\phi\in C_c^\infty(\R^n)\,. 
\end{align}
After these preparations we proceed with the local estimate.
Fix a ball $B(x_\ell,r_\ell)$, $\ell\in\{0,\ldots, M\}$,  with $x_\ell\in\Omega$ and define
in view of~\bref{Eudecomp} and~\bref{ueqn}
the vector field $u_f=(u_f^{(1)}\,\ldots, u_f^{(n)})$ by
\begin{align}\label{eqdefuf}
 u_f^{(i)} = -\sum_{j=1}^nD_j N\ast\bigl(  (2f_{ij}-\delta_{ij}(\Tr f))
\chi_{B(x_\ell,r_\ell)}\bigr)\,.
\end{align}
Analogously we set
\begin{align*}
 u_g^{(i)} = -\sum_{j=1}^n D_j N\ast \bigl( (2 g_{ij}-\delta_{ij}(\Tr g))
\chi_{B(x_\ell,r_\ell)}\bigr)\,.
\end{align*}
Then $\|Du_f\|_p\leq A_p \| f\|_p$ and $\|Du_g\|_p\leq A_p \| g\|_p$. Moreover, 
$u_f$ and $u_g$ are locally weak solutions of~\bref{veryweak} in the sense that we have for all 
$\phi\in C_c^\infty( B(x_\ell,r_\ell) )$
the identities
\begin{align*}
-  \int_{\R^n} u^{(i)}_f \Delta \phi \dv{x}
=\int_{\R^n} \sum_{j=1}^n ( 2f_{ij}-\delta_{ij}(\Tr f))
D_j\phi \dv{x}
\end{align*}
and
\begin{align*}
 - \int_{\R^n} u^{(i)}_g \Delta \phi \dv{x}
=\int_{\R^n} \sum_{j=1}^n (2 g_{ij}-\delta_{ij}(\Tr g))
D_j\phi \dv{x}\,. 
\end{align*}
In view of~\bref{Eudecomp} and~\bref{ueqn} we see that the function $w=u-u_f-u_g$ defines a harmonic distribution
on $B(x_\ell,r_\ell)$ which
can be identified with a smooth harmonic function by Weyl's lemma.
By Sobolev's embedding theorem and Caccioppoli estimates for harmonic functions 
we infer for the harmonic function $Ew$ that
\begin{align*}
  \|Ew\|_{L^q(B(x_\ell,r_\ell/2))} &\le c
  \|Ew\|_{L^p(B(x_\ell,r_\ell))} = c
  \|Eu-E\uf-E\ug\|_{L^p(B(x_\ell,r_\ell))} \\
&\le c \| f  + g \|_{L^p(B(x_\ell,r_\ell))} + c\|D\uf\|_{L^p(B(x_\ell,r_\ell))}
+ c\|D\ug\|_{L^p(B(x_\ell,r_\ell))}\\
&\le c  \| f  \|_{L^p(B(x_\ell,r_\ell))} +c  \|g\|_{L^q(B(x_\ell,r_\ell))} \,.
\end{align*}
In this estimate the constant may depend on $r$ and hence on $\Omega$.
By Korn's inequality in $L^q$ (see  (\ref{eqkornlp}))
applied to the set $B(x_\ell,r_\ell/2)$
there is a skew-symmetric matrix $S_\ell\in\R^{n\times n}$ such that
\begin{align*}
  \|Dw-S_\ell\|_{L^q(B(x_\ell,r_\ell/2))} \le c   \|Ew\|_{L^q(B(x_\ell,r_\ell/2))} \,.
\end{align*}
Setting $F=D\uf$ and $G=D\ug+Dw- S_\ell$ we conclude
\begin{align*}
  Du=F + G +S_\ell\qquad a.e.\text{ in }B(x_\ell,r_\ell/2)
\end{align*}
with
\begin{align*}
  \|F\|_{L^p(B(x_\ell,r_\ell/2))}\le c\| f \|_{L^p(B(x_\ell,r_\ell))}\,,\quad
  \|G\|_{L^q(B(x_\ell,r_\ell/2))}
\le c\|g\|_{L^q(B(x_\ell,r_\ell))}+c\| f \|_{L^p(B(x_\ell,r_\ell))}\,.
\end{align*}
The additional term in the estimate for 
$G$ will be replaced in the global estimate by~\bref{order}.

\medskip

\noindent
{\em Step 4: Local estimate at the boundary.}
Fix a ball $B(x_\ell,r_\ell)$, $\ell\in\{0,\ldots, M\}$, with $x_\ell\in\partial\Omega$. 
By construction, $B(x_\ell,r_\ell)$ satisfies the hypotheses of the extension
theorem (Theorem \ref{theo:extension} below)
and therefore $u$, $ f $ and $g$ 
have extensions $\uext$, $\fext$ and $\gext$
which are defined on $\Omega\cup B(x_\ell,\gamma r_\ell)$ with $E\uext =\fext+\gext$ a.e.~on 
$B(x_\ell,\gamma r_\ell)$.
Moreover, these functions satisfy the estimates~\bref{extest}.
We define $\uftilde$ and $\ugtilde$ as in Step~3
and proceed as before to obtain $S_\ell$, $\Fext$, $\Gext$ such that
\begin{align*}
  D\uext=\Fext+\Gext+S_\ell \,,\qquad a.e.\text{ in }B(x_\ell,\gamma r_\ell/2)
\end{align*}
and
\begin{align*}
  \|\Fext\|_{L^p(B(x_\ell,\gamma r_\ell/2))}&\, \le c\|\fext\|_{L^p(B(x_\ell,\gamma r_\ell))}
\le c\| f \|_{L^p(B(x_\ell,r_\ell)\cap\Omega)}\,,\\
  \|\Gext\|_{L^q(B(x_\ell,\gamma r_\ell/2))}
&\,\le c\|\gext\|_{L^q(B(x_\ell,\gamma r_\ell))}+c\|\fext\|_{L^p(B(x_\ell,\gamma r_\ell))}\\
&\, \le c\|g\|_{L^q(B(x_\ell,r_\ell)\cap\Omega)}+c\| f \|_{L^p(B(x_\ell,r_\ell)\cap\Omega)}\,.
\end{align*}
We define $F$ and $G$ to be the restrictions of $\Fext$ and $\Gext$ to 
$B(x_\ell,\gamma r_\ell/2)\cap\Omega$.

\medskip

\noindent
{\em Step 5: Global estimate.}
Let $S_i$, $F_i$, $G_i$, $i=0, \ldots, M$,  
be the matrices and fields in the balls $B(x_i, \gamma r_i/2)$ which
were constructed in Steps~3 and~4, respectively. In order to prove the assertion, we need
to verify that we may choose $S_0$ globally in $\Omega$. Therefore
we estimate $|S_i-S_j|$. If $B_i\cap B_j\cap \Omega\ne \emptyset$, then
recalling (\ref{eqdefalpha}) we obtain
\begin{align*}
\alpha  |S_i-S_j|
&\, \leq \int_{B_i\cap B_j\cap\Omega}|S_i-S_j|\dv{y}  \\
&\, \leq \int_{B_i\cap\Omega}|S_i-Du|\dv{y} + \int_{B_j\cap\Omega}|S_j-Du|\dv{y}
\le c \| f \|_{L^p(\Omega)} +c\|g\|_{L^q(\Omega)}\,.
\end{align*}
Since $\Omega$ is connected and the subcover consists of finite number of balls, 
we infer
\begin{align*}
  |S_i-S_0|\le c \| f \|_{L^p(\Omega)} +c\|g\|_{L^q(\Omega)}
\end{align*}
for all $i=1,\ldots, M$ where $c$ depends only on $p$, $q$, $n$ and
$\Omega$. We define  inductively
a family $A_i$, $i=0,\ldots, M$, of pairwise disjoint and measurable sets by $A_0=B_0\cap\Omega$ and
\begin{align*}
 A_i =\Bigl(  B_i\setminus \bigcup_{j=0}^{i-1} A_j\Bigr)\cap\Omega\,,\quad i\geq 1\,,
\end{align*}
which covers $\Omega$ up to a set of measure zero and set
\begin{align*}
F=\sum_{i=0 }^M F_i \chi_{A_i}\,,\hskip1cm
G=\sum_{i=0 }^M (G_i + S_i-S_0) \chi_{A_i}\,.
\end{align*}
We obtain the decomposition~\bref{korndecomposition} and the corresponding estimates, i.e.,
\begin{align*}
  Du-S_0= F+ G \,,\hskip5mm
  \|F\|_{L^p(\Omega)}\le c\| f \|_{L^p(\Omega)}\,,\hskip5mm
  \|G\|_{L^q(\Omega)}\le c\|g\|_{L^q(\Omega)}+c\| f \|_{L^p(\Omega)}\,.
\end{align*}
The assertion follows from~\bref{order}. The proof is now complete.
\end{proof}


\section{Nonlinear estimate: geometric rigidity and proof of Theorem~\ref{theo:rigiditypqnonlinear}}
\label{secnonlinear}

In this section we prove Theorem~\ref{theo:rigiditypqnonlinear}.
We start from the case that $u$ is globally Lipschitz continuous. 
The general case will be reduced to this situation via a truncation argument
which provides a Lipschitz constant which depends only on $n$ and $\Omega$.

\begin{lemma}\label{lemmalipschitzcase}
Theorem~\ref{theo:rigiditypqnonlinear}   holds under  the additional assumption that
there exits a constant $M>n$ so that  $|D u|\le M$ a.e. The constant 
in the estimates~(\ref{thmestimates}) depends on $M$. 
\end{lemma}

\begin{proof}
The proof is based on a suitable linearization at the identity. In order to control the
higher order terms we first assume that $q$ is not much larger than $p$.
In the following, all constants may depend on $n$, $p$, $q$, $M$, and $\Omega$.

We may assume that $0\le f,g\le 2M$. Indeed, if this
is not the case we replace $f$ and $g$ by $f'=\min\{|f|, 2M\}$ and 
$g'=\min\{|g|, 2M\}$, respectively, and estimate $\dist(Du, \SO(n))\le
\min\{ f+g, |Du|+\sqrt n\}\le \min\{|f|+|g|, M+\sqrt n\}\le f'+g'$.

\medskip

\noindent
{\em Step 1. Small $q$.}
Assume first that $q\le 2p$. 
We observe that the assumption $|f|\leq 2M$ implies $\|f\|_q\le c \|f\|_p^{p/q}$.
  By the rigidity estimate in $L^q$, see
  \cite[Th. 3.1]{FrieseckeJamesMueller2005} and \cite[Sect. 2.4]{ContiSchweizer2006b}, 
there exists a $Q\in \SO(n)$ such that in view of~\bref{estimatedist}
\begin{align}\label{contirigidity}
    \|D u-Q\|_q \le c \|\dist(D u, \SO(n))\|_q \le c \|f\|_q
    + c\|g\|_q \le c \|f\|_p^{p/q}
    + c\|g\|_q  \,.
 \end{align}
If $ \|f\|_p^{p} \le \|g\|_q^q$, then the assertion follows with $F=0$ and $G=Du-Q$.
We may thus assume that
\begin{align}\label{fgorder}
 \|f\|_p^{p} \ge \|g\|_q^q
\end{align}
and compose $u$ with a rotation so that $Q=\Id$. We expand
 the distance to $\SO(n)$ in the identity matrix  and  obtain the pointwise estimate
  \begin{equation}\label{eqtaylor}
    |Eu-\Id|\le c\, \dist(D u, \SO(n)) + c |D u
    -\Id|^2 \qquad a.e.\text{ in }\Omega\,.
  \end{equation}
Observe that in view of the $L^\infty$  bound on $|Du|$
and the condition $q\le 2p$
\begin{align}\label{expansion}
    \left\| |D u -\Id|^2 \right\|_p^p = \int_\Omega |D u-\Id|^{2p} \dv{x}
    \le c \int_\Omega |D u-\Id|^q  \dv{x}\,, 
\end{align}
and that~\bref{expansion},~\bref{contirigidity}, and~\bref{fgorder}
imply
\begin{equation*}
  \left\| |D u -\Id|^2 \right\|_p^p \le  c \|f\|_p^{p}
    + c\|g\|_q^q \le c \|f\|_p^{p} \,.
\end{equation*}
Let $ \widetilde f =f+ |D u -\Id|^2$, $\widetilde g =g$. By the foregoing estimate,
$ \|\widetilde f \|_p \le c \|f\|_p$, and 
(\ref{eqtaylor}) gives
\begin{equation*}
  |Eu-\Id|\le c \widetilde f   + \widetilde g  \,.
\end{equation*}
The assertion follows now from Theorem \ref{theokorn}. 
Indeed, there exists a skew-symmetric matrix $S$ and matrix fields $\Ftilde$ and
$\Gtilde$ such that $D u -\Id = S+ \Ftilde +\Gtilde$ with $\|\Ftilde\|_p\leq c\|\widetilde f\|_p
\leq c\|f\|_p$
and $\|\Gtilde\|_q\leq c\|\widetilde g\|_q=c\|g\|_q$. Let $Q\in \SO(n)$ be a proper rotation such that
$|\Id+S-Q|=\dist(\Id+S, \SO(n))$. Then 
\begin{align*}
 |\Id+S - Q | 
\leq |\Id+S-D u| + \dist(D u,\SO(n))
\leq |\Ftilde|+|\Gtilde| + |f| + |g|\,.
\end{align*}
Combining with the previous estimates we obtain
$|\Id+S - Q | \le c\|f\|_p+ c\|g\|_q$.
If $\|f\|_p\le \|g\|_q$ we set $F=\Ftilde$, $G=\Gtilde+Q-\Id -S$, otherwise
 $F=\Ftilde+Q-\Id -S$, $G=\Gtilde$.

\medskip

\noindent
{\em Step 2: Large $q$.} Let 
\begin{equation*}
  \Lambda_k=\{(p,q)\in(1,\infty)^2: 2^kp<q\leq 2^{k+1}p\}\,.
\end{equation*}
We shall prove by induction on $k\in\N$ the following assertion.
For every $(p,q)\in \Lambda_k$ and every $u$ as in the statement 
 there
exist a rotation $Q\in\SO(n)$ and matrix fields $F\in L^p(\Omega;\MM{n}{n})$ 
and $G\in L^q(\Omega;\MM{n}{n})$ such that $Du = Q+F+G$ a.e., 
\begin{align}\label{inductionestimate}
 \| F\|_{L^p(\Omega)} \leq c_k  \| f\|_{L^p(\Omega)}\,,\qquad
 \| G\|_{L^q(\Omega)} \leq c_k  \| g\|_{L^q(\Omega)}\,,
\end{align}
and\begin{align}\label{uniformbound}
|F|\leq M+\sqrt{n}\,,\qquad|G|\leq M+\sqrt{n}\,,
\end{align}
where the constant $c_k$ depends only on $n$, $p$, $q$, $M$, $k$ and $\Omega$.
Notice that it suffices to prove (\ref{inductionestimate}), then
(\ref{uniformbound}) follows. Indeed, 
let $A=\{x\in \Omega\colon |F|>M+\sqrt{n}\}$ and
$B=\{x\in \Omega\colon |G|>M+\sqrt{n}\}$. Then it suffices to replace $F$ and
$G$ by 
\begin{align*}
 F'= \chi_{\Omega\setminus A}F+\chi_A(Du-Q)\,,\qquad
G'= \chi_{\Omega\setminus(A\cup B)}G+\chi_{B\setminus A}(Du-Q)\,,
\end{align*}
respectively.

The case
$k=0$ has been verified in Step~1. Suppose thus that the assertion has been proven for
$k\geq 0$ and that  $(p,q)\in \Lambda_{k+1}$. Then 
 $(2p,q)\in \Lambda_k$ and by assumption there exist
a rotation $Q$ and matrix fields $F$ and $G$ which satisfy $Du = Q+F+G$
and the estimate~\bref{inductionestimate}, that is,
\begin{equation}\label{reductionnorms}
  \|F\|_{{2p}} \le c\|f\|_{{2p}} \le c\|f\|_{{p}}^{1/2}\,,\hskip1cm 
  \|G\|_{q} \le c\|g\|_{q}\,.
\end{equation}
As above, we may assume that $Q=\Id$ and use the Taylor series~\bref{eqtaylor} to 
obtain the estimates
\begin{equation*}
  |E u-\Id| \le c f+c g+c |F|^2 + c|G|^2\qquad a.e.\text{ in }\Omega
\end{equation*}
and in view of~\bref{uniformbound} and~\bref{reductionnorms}
\begin{equation*}
  \|\,|F|^2\,\|_p = \|F\|_{2p}^2 \le c \|f\|_p \,,\hskip1cm
  \|\,|G|^2\,\|_q \le c \|G\|_q\le c\|g\|_{q}\,.
\end{equation*}
Therefore the assertion follows  from Theorem \ref{theokorn}
with $\widetilde f =f+|F|^2$, $\widetilde g =g+|G|^2$.
\end{proof}

In order to prove Theorem   \ref{theo:rigiditypqnonlinear}  we make use of a
well-known truncation result.

\begin{theorem}\label{theo:truncation}
  Let $\Omega\subset\R^n$ be a bounded Lipschitz domain, $m\geq 1$. Then there
  is a constant $\setc{eg}=\const{eg}(\Omega)\geq 1$ such that for all $u\in W^{1,1}(\Omega;\R^m)$
  and all $\lambda>0$ there exists  a measurable set $E\subset\Omega$ such that:
  \begin{enumerate}
  \item $u$ is $\const{eg}\lambda$-Lipschitz on $E$\,;
  \item\label{theo:truncation2} $\displaystyle|\Omega\setminus E|\le
    \frac{\const{eg}}{\lambda}\int_{\{ |Du|>\lambda\}} |Du| \dv{x}$.

  \end{enumerate}
\end{theorem}

\begin{proof}
This result corresponds essentially to Proposition~A.1 in~\cite{FrieseckeJamesMueller2005} 
and follows from the
  same proof. The techniques are analogous to the proof in the case $\Omega=\R^n$ 
in Section~6.6.3 in~\cite{EvansGariepy1992}.
\end{proof}

\begin{proof}[Proof of Theorem~\ref{theo:rigiditypqnonlinear}]
Suppose first that  $\|g\|_q\le \|f\|_p$. Then by~\bref{estimatedist}
\begin{align*}
  \|\dist(D u, \SO(n))\|_p\le 
\|g\|_p+ \|f\|_p \le 
c \|g\|_q+ \|f\|_p \le (c+1)\|f\|_p
\end{align*}
and the result follows from the geometric rigidity estimates in
$L^p$ with $F=Du-Q$ and  $G=0$ for a suitable $Q\in\SO(n)$.

We may thus assume that $\|g\|_q\ge \|f\|_p$. In order to apply Lemma~\ref{lemmalipschitzcase}
we choose in Theorem~\ref{theo:truncation} $\lambda=2n$ and obtain a measurable set $E$ such that
$u$ is Lipschitz continuous on $E$ with Lipschitz constant $M=2n\const{eg}$.
Let $u_M$ be a Lipschitz-extension of $u|_E$ to $\Omega$ with the same 
Lipschitz constant which exists according to Kirszbraun's 
Theorem~\cite[Section~2.10.43]{Federer1996}. 
Then $u_M$ is $M$-Lipschitz and $u_M=u$ on $E$.
We define
 \begin{equation*}
f_M=f \text{ and } g_M=
g+2M\chi_{\Omega\setminus E} \text{ if } \|f\|_p^p\le \|g\|_q^q
\end{equation*}
and
 \begin{equation*}
f_M=f+2M\chi_{\Omega\setminus E} \text{ and } g_M=
g \text{ if }  \|g\|_q^q<\|f\|_p^p
\end{equation*}
and assert that
\begin{equation}\label{claimuM}
  \dist(Du_M,\SO(n)) \le f_M + g_M \text{ a.e. in }\Omega\,.
\end{equation}
For almost every  $x\in E$ we have $Du_M=Du$, $f_M=f$ and $g_M=g$, hence 
(\ref{claimuM}) holds. For almost every $x\in\Omega\setminus E$ we have
\begin{equation*}
    \dist(Du_M,\SO(n))(x)\le |M|+\sqrt n \le 2M \le 2M\chi_{\Omega\setminus
      E}(x) \le  (f_M+g_M)(x)\,.
\end{equation*}
This proves (\ref{claimuM}).

We now assert that
\begin{equation}\label{claimtildefp}
  \|f_M\|_p \le C(\Omega,p) \|f\|_p\,, \text{ and }
  \|g_M\|_q \le C(\Omega,q) \|g\|_q\,.
\end{equation}
To see this we first estimate $|\Omega\setminus E|$. We have for $A\in \R^{n\times n}$ with 
$ |A|\ge   2n$
\begin{equation*}
  |A| \le \sqrt{n}+\dist(A, \SO(n)) \le 2\,\dist (A, \SO(n))\,. 
\end{equation*}
Therefore
\begin{equation*}
  |Du|\le \max\{4f,4g\} \text{ if } |Du|\ge \lambda=2n\,.
\end{equation*}
Together with Theorem \ref{theo:truncation}\ref{theo:truncation2}
this yields 
\begin{alignat*}1
  |\Omega\setminus E| &\le \frac{c_1}{\lambda} \int_{\{|Du|>\lambda\}}
  |Du| 
  \dv{x} \\
& \le\frac{c_1}{\lambda}  \int_{\{4f\ge \lambda\}} 4 f  \dv{x} +
\frac{c_1}{\lambda} \int_{\{4g\ge \lambda\}} 4 g  \dv{x} \\
&\le \frac{4^p c_1}{\lambda^p} \int_\Omega f^p \dv{x}+
\frac{4^q c_1}{\lambda^q} \int_\Omega g^q \dv{x}
\le c \left( \|f\|_p^p+\|g\|_q^q \right) \,.
\end{alignat*}
If $\|f\|_p^p\le \|g\|_q^q$ then 
 $\|2M\chi_{\Omega\setminus E} \|_q^q 
\le (2M)^q 2c \|g\|_q^q$ and therefore $\|g_M\|_q\le c \|g\|_q$. 
The other case is analogous. This concludes the proof of 
 (\ref{claimtildefp}).

It follows from  (\ref{claimuM}), (\ref{claimtildefp})  and Lemma
\ref{lemmalipschitzcase} that there exist an $R\in\SO(n)$ and matrix fields
$ F_M$, $G_M$ such that
\begin{equation}\label{equmtildefg}
  Du_M-R= F_M+ G_M,\hskip2mm
  \| F_M\|_p \le c\| f_M\|_p \le c\|f\|_p\,,\hskip2mm
  \| G_M\|_q \le c\|g_M\|_q \le c\|g\|_q\,.
\end{equation}
Now $|Du-Du_M|\le |Du|+M\le \dist(Du, \SO(n)) + \sqrt{n}+M$ almost everywhere on $\Omega$ and $Du=Du_M$
almost everywhere on $E$. Thus $|Du-Du_M|\le  f_M+g_M$ and $|Du-R|\le 
f_M+ |F_M|+|G_M|+g_M$ and the assertion follows from
(\ref{claimtildefp}), (\ref{equmtildefg}) and (\ref{getdecomposition}).
\end{proof}

\section{Applications and extensions}
\label{weakspaces}
\subsection{Estimates in Lorentz  spaces}
As an application of our estimates in Section~\ref{secnonlinear} we present rigidity results
in the Lorentz spaces $L^{p,q}$
for $p\in (1,\infty)$, $q\in[1,\infty]$, see 
\cite{Peetre1963,Hunt1966,ButzerBerens1967,Lunardi2009,TartarInterp}. 
For $q=\infty$ they coincide with  the weak $L^p$ spaces.
The Lorentz space $L^{p,q}(\Omega)$ is equal to the 
real interpolation space which is constructed with the $K$-functional, 
\begin{align*}
 L^{p,q}(\Omega)=\bigl( L^{p_1}(\Omega), L^{p_2}(\Omega)\bigr)_{\theta, q}
\end{align*}
where the  $K$ functional is given by
\begin{align*}
 K(w, t) = \inf \left\{ \| f\|_{p_1}+ t \|g\|_{p_2}:
w=f+g,\, f\in L^{p_1},\, g\in L^{p_2}\right\}
\end{align*}
and
\begin{align}\label{interpolationcondition}
1\leq p_1<p_2\leq \infty, \, 1\leq q\leq \infty, \, \theta\in (0,1),\,
\frac{1}{p}=\frac{1-\theta}{p_1}+\frac{\theta}{p_2}\,.
\end{align}
The norm is given for $q<\infty$ by 
\begin{equation*}
  \|w\|_{p,q} =\left(\int_{(0,\infty)} (t^{-\theta} K(w,t))^q \frac{\dv{t}}{t}
  \right)^{1/q} \,,
\end{equation*}
in the special case $q=\infty$ by

\begin{align*}
\|w\|_{\theta,\infty} = \sup_{t>0} t^{-\theta} K(w, t)\,.
\end{align*}
We remark that different choices of $\theta, p_1, p_2$ which satisfy
(\ref{interpolationcondition})  give equivalent norms. 
In this framework we obtain the following result.
\begin{corollary}
 Suppose that $p\in(1,\infty)$, $q\in[1,\infty]$. Then there exists a constant
 $c$ which only depends on $p$, $n$ 
and $\Omega$ such that the following assertion is true. If $u\in
W^{1,1}(\Omega;\R^n)$ with  
$\dist(Du,\SO(n))\in L^{p,q}(\Omega;\MM{n}{n})$, then there 
exists a rotation $Q\in \SO(n)$ such that
\begin{align*}
 \| Du - Q\|_{p,q} \leq c\| \dist(Du,\SO(n))\|_{p,q}\,.
\end{align*}
\end{corollary}
\begin{proof}
Fix any triple $\theta$, $p_1$, $p_2$ with $1<p_1<p_2<\infty$  which
satisfies~\bref{interpolationcondition} and 
set $w=\dist\bigl( Du,\SO(n)\bigr)$. By assumption,
$K(w,t)<\infty$ for almost all $t$. 
Hence there exists for almost all $t>0$ a decomposition $w=f_t+g_t$ with 
$f_t\in L^{p_1}(\Omega)$, $g_t\in L^{p_2}(\Omega)$ and
\begin{align*}
 K(w,t)=  \| f_t\|_{p_1} +  t \|g_t\|_{p_2}\,.
\end{align*}
In view of the rigidity estimate in $L^{p_1}+L^{p_2}$, Theorem \ref{theo:rigiditypqnonlinear},
 there exist a rotation $Q_t\in \SO(n)$ and
matrix fields $F_t\in L^{p_1}(\Omega;\MM{n}{n})$, $G_t\in L^{p_2}(\Omega;\MM{n}{n})$ with 
\begin{align*}
 Du = Q_t + F_t + G_t\,,\quad \|F_t\|_{p_1}\leq c\|f_t\|_{p_1}\,,\quad \|G_t\|_{p_2}\leq c\|g_t\|_{p_2}\,.
\end{align*}
We define 
\begin{align*}
 F_t'=F_t-\frac{1}{|\Omega|} \int_\Omega F_t \dv{x}\,,\quad 
G_t'=G_t-\frac{1}{|\Omega|} \int_\Omega G_t \dv{x}\,,\quad 
R=\frac{1}{|\Omega|} \int_\Omega Du \dv{x}
\end{align*}
and obtain
\begin{align*}
 Du = R+ F_t' + G_t'\,,\quad 
\|F_t'\|_{p_1}\le 2\|F_t\|_{p_1}\,,\quad 
\|G_t'\|_{p_2}\le 2\|G_t\|_{p_2}\,.
\end{align*}
Therefore  for almost all
$t>0$  one has
\begin{align*}
 Du = R + F_t' + G_t'\,,\quad \|F_t'\|_{p_1}
+t\|G_t'\|_{p_2}\leq 
c\|f_t\|_{p_1}
+ct\|g_t\|_{p_2}=cK(w,t)\,,
\end{align*}
which implies $K(Du-R,t)\le c K(w,t)$ for almost all $t$. We stress that the constant
does not  depend on $t$ but only on $p_1$, $p_2$, $\Omega$. We conclude that
\begin{equation*}
  \|Du-R\|_{p,q}\le c \|w\|_{p,q}\,.
\end{equation*}
From $\dist(R,\SO(n))\le \dist(Du,\SO(n))+|Du-R|$  we obtain
that there is a rotation $Q$ with $|Q-R|\le c\|w\|_{p,q}$, and conclude
\begin{equation*}
  \|Du-Q\|_{p,q}\le 
  \|Du-R\|_{p,q}+\|Q-R\|_{p,q}
\le c  \|\dist(Du, \SO(n))\|_{p,q}\,.
\end{equation*}
This concludes the proof.
\end{proof}

\subsection{Equiintegrability}
As a second application of our work we show that equiintegrability of the
distance from the set of rotations implies equiintegrability of the distance
from a fixed rotation. 
A sequence $f_k\in  L^p(\Omega)$, $k\in\N$, is
  $L^p$-equiintegrable if for every   $\epsilon>0$ there is a $\delta>0$
  such that for all measurable sets $E\subset\Omega$   with $|E|<\delta$ one has $\int_E
  |f_k|^p \dv{x}\le \epsilon$ for all $k\in\N$.
For bounded sets  $\Omega\subset\R^n$ this is equivalent to the fact that
      \begin{equation}\label{eqequiintegr}
        \lim_{T\to\infty} \sup_{k\in\N} \int_{\{|f_k|>T\}} |f_k|^p \dv{x} =0\,,
      \end{equation}
see for example \cite[Theorem 2.29, page 151]{FonsecaLeoniLpspaces} for a proof.

The following statement generalizes the assertion in \cite[page
221]{FrieseckeJamesMueller2008} concerning the interplay of equiintegrability
and rigidity.
\begin{corollary}\label{corollaryequi}
  Let $\Omega\subset\R^n$ be a
  connected and bounded Lipschitz set. Consider a sequence of positive numbers
  $\eta_k\in(0,\infty)$ and a sequence of functions $u_k\in
  W^{1,p}(\Omega;\R^n)$. Assume that  the sequence
  \begin{equation*}
    d_k=\eta_k\,\dist(Du_k,\SO(n))
  \end{equation*}
is $L^p$-equiintegrable. Then
there are rotations $Q_k\in\SO(n)$ such that
\begin{equation*}
z_k=\eta_k\,(Du_k-Q_k)
\end{equation*}
 is $L^p$-equiintegrable.   
\end{corollary}
The corresponding linear result can be proven exactly in the same way, using
Theorem \ref{theokorn} instead of Theorem \ref{theo:rigiditypqnonlinear}. 
We further remark that for bounded sequences $(\eta_k)$  the statement follows
immediately from the boundedness of $\SO(n)$, the interesting case is
$\eta_k\to\infty$. 
\begin{proof}
We shall use the characterization of equiintegrability given in
(\ref{eqequiintegr}).
Pick some $\epsilon>0$. Then there is $T_\epsilon$ such that
  \begin{equation*}
    \int_{\{d_k>T_\epsilon\}} d_k^p \dv{x}\le \epsilon
  \end{equation*}
  for all $k$. We define
  \begin{equation*}
    f_k=\dist(Du_k, \SO(n))\chi_{\{d_k>T_\epsilon\}} \text{ and }
    g_k=\dist(Du_k, \SO(n))\chi_{\{d_k\le T_\epsilon\}}
  \end{equation*}
and observe that $\|\eta_k f_k\|_p^p \le \epsilon$ and $\|\eta_k g_k\|_\infty\le T_\epsilon$. 
  We fix some $q\in(p,\infty)$, for example $q=p+1$, and estimate
  \begin{equation*}
    \|\eta_kg_k\|_q^q \le T_\epsilon^{q-p} \|d_k\|_p^p\le M T_\epsilon^{q-p} \,,
  \end{equation*}
  where $M=\sup_k \|d_k\|_p^p$. Notice that $M<\infty$, since $L^p$-equiintegrable
  sequences are bounded in $L^p$.

By Theorem \ref{theo:rigiditypqnonlinear} applied to each $u_k$ there are 
rotations $Q_k$ and fields $F_k$, $G_k$ such that
\begin{equation*}
  Du_k=Q_k+F_k+G_k\,,\hskip5mm
  \|\eta_kF_k\|_p^p\le c \|\eta_k f_k\|_p^p \le c\epsilon\,,\hskip5mm
  \|\eta_k G_k\|_q^q\le c\|\eta_k g_k\|_q^q\,.
\end{equation*}
We set  $E_k=\{\eta_k|G_k|>L_\epsilon\}$ for some  $L_\epsilon>0$ to be chosen later, 
estimate its measure
\begin{equation*}
 |E_k|\le \frac{1}{L_\epsilon^q} \int_{E_k} |\eta_kG_k|^q \dv{x}
 \le \frac{c}{L_\epsilon^q} \|\eta_kg_k\|_q^q
\end{equation*}
and the $L^p$ norm on $E_k$ via
\begin{equation*}
  \|\eta_k G_k\|_{L^p(E_k)} \le |E_k|^{1/p-1/q}   \|\eta_k G_k\|_{L^q(E_k)} \,.
\end{equation*}
We conclude that 
\begin{equation*}
  \int_{E_k} |\eta_k G_k|^p \dv{x}\le |E_k|^{(q-p)/q}   \|\eta_k g_k\|_q^p
  \le 
\frac{c}{L_\epsilon^{q-p}}   \|\eta_kg_k\|_q^{q}
\le c M \frac{T_\epsilon^{q-p}}{L_\epsilon^{q-p}}\,.
\end{equation*}
We finally choose $L_\epsilon=T_\epsilon /\epsilon^{1/(q-p)}$, so that the last fraction
is smaller than $\epsilon$ and obtain for  $z_k=\eta_k(Du_k-Q_k)=\eta_kF_k+\eta_kG_k$, setting
$E'_k=\{|z_k|>2L_\epsilon\}$, 
\begin{alignat*}1
  \int_{E'_k} |z_k|^p \dv{x}
&=   \int_{E'_k\cap \{|F_k|\ge|G_k|\}} |z_k|^p\dv{x}
+   \int_{E'_k\cap \{|F_k|<|G_k|\}} |z_k|^p\dv{x}\\
&  \le   2^p \|\eta_kF_k\|^p_p + 2^p\int_{E_k} |\eta_kG_k|^p \dv{x}
\le c\epsilon + cM\epsilon\,.
\end{alignat*}
Therefore for all $t>2L_\epsilon$ and all $k$ we have $\int_{\{|z_k|>t\}}|z_k|^p\dv{x}\le
c(1+M)\epsilon$, and (\ref{eqequiintegr}) is proven.
\end{proof}

\subsection{Multiple exponents}
In this section we generalize our results to the case of more than two exponents. We first establish 
the linear estimate which is parallel to Theorem \ref{theokorn} and then indicate how
this estimate implies the corresponding nonlinear estimate.

\begin{theorem}\label{theo:multikorn}
  Let $\Omega\subset\R^n$ be a bounded and connected domain with  Lipschitz boundary.
Suppose that $1<p_1<p_2<\ldots <p_N<\infty$, $N\ge 1$, and that $u\in
W^{1,1}(\Omega;\R^n)$, $f_\alpha \in 
  L^{p_\alpha}(\Omega;\R^{n\times n})$, $\alpha=1, \ldots, N$, 
are given with
\begin{equation*}
Eu=\frac{1}{2}\bigl(  D u + D u^T \bigr) 
=   \sum_{\alpha=1}^N f_\alpha  \qquad a.e.\text{ in }\Omega\,.
\end{equation*}
Then there exists a constant $c$, matrix fields $F_\alpha\in
L^{p_\alpha}(\Omega;\R^{n\times n})$, $\alpha=1, \ldots, N$, and a skew-symmetric matrix  
$S\in \R^{n\times n}$ such that
\begin{equation*}
  D u = S + \sum_{\alpha=1}^N F_\alpha \qquad a.e.\text{ in }\Omega\,,
\end{equation*}
and
\begin{equation*}
  \|F_\alpha\|_{L^{p_\alpha}(\Omega)}\le c \| f_\alpha  \|_{L^{p_\alpha}(\Omega)}\,,\quad\alpha=1, \ldots, N\,.
\end{equation*}
The constant depends only on $n$, $\Omega$, $p_1,p_2,\dots, p_N$.
\end{theorem}
\begin{proof}
  We follow the scheme of the proof of Theorem \ref{theokorn}.
  By induction we can assume that the statement has been proven for $N-1$ exponents with $N\geq 3$, the assertion for $N=2$ is the
statement in Theorem \ref{theokorn}.
  If $\|f_\alpha\|_{p_\alpha}\ge \|f_{\alpha+1}\|_{p_{\alpha+1}}$ for some $\alpha\in\{ 1, \ldots, N-1\}$, then 
we eliminate the exponent $\alpha+1$,
  define $\tilde f_\alpha=f_\alpha+f_{\alpha+1}$, observe
  $\|\tilde f_\alpha\|_{p_\alpha}\le \|f_\alpha\|_{p_\alpha}+c\|f_{\alpha+1}\|_{p_{\alpha+1}}
  \le c \|f_\alpha\|_{p_\alpha}$ and apply the statement to the remaining $N-1$
  exponents and the corresponding $N-1$ functions $f_1, f_2, \dots, f_{\alpha-1}, \tilde f_\alpha, f_{\alpha+2},
  \dots, f_N$. 

  Therefore we may assume that
  \begin{equation*}
    \|f_\alpha\|_{p_\alpha}\le \|f_{\alpha+1}\|_{p_{\alpha+1}} \text{ for all } \alpha\in \{1, \ldots, N\}\,.
  \end{equation*}
  Steps 1 and 2 in the proof of Theorem \ref{theokorn} are unchanged. In Step 3 we define
vector fields  $u_\alpha$ associated to the matrix fields $f_\alpha$ as in (\ref{eqdefuf}). Then
  $\|Du_\alpha\|_{p_\alpha}\le A_{p_\alpha}\|f_\alpha\|_{p_\alpha}$, the
  function
$w=u-\sum_{\alpha=1}^N u_\alpha$ is harmonic, and satisfies
\begin{align*}
  \|Ew\|_{L^{p_N}(B(x_\ell,r_\ell/2))} &\le c
  \|Ew\|_{L^{p_1}(B(x_\ell,r_\ell))} = c
  \bigl\|Eu-\sum_{\alpha=1}^N Eu_\alpha\bigr\|_{L^{p_1}(B(x_\ell,r_\ell))} \\
&\le c \bigl\| \sum_{\alpha=1}^N f_\alpha \bigr\|_{L^{p_\alpha}(B(x_\ell,r_\ell))} + c
\sum_{\alpha=1}^N \|D u_\alpha\|_{L^{p_\alpha}(B(x_\ell,r_\ell))}
\le c \sum_{\alpha=1}^N \| f_\alpha  \|_{L^{p_\alpha}(B(x_\ell,r_\ell))}  \,.
\end{align*}
We  apply Korn's inequality in $L^{p_N}$ to the ball $B(x_\ell,r_\ell)$, 
and obtain the analogous decomposition of $Du$ together
with the estimates. An extension theorem  corresponding to
Theorem \ref{theo:extension}  holds with $N$ terms and
with the same proof, Steps 4 and 5 can be
concluded as before with minor notational changes. 
\end{proof}
We now turn to the nonlinear estimate.
\begin{theorem}\label{theo:rigiditypqnonlinearmanyexponents}
  Let $\Omega\subset\R^n$ be a bounded and connected domain with Lipschitz boundary. Suppose
that  $1<p_1<p_2<\ldots<p_N<\infty$ and that $u\in W^{1,1}(\Omega;\R^n)$,
$f_\alpha\in L^{p_\alpha}(\Omega)$, $\alpha=1,2,\dots, N$ are given with
\begin{align*}
   \dist(D u, \SO(n)) = \sum_{\alpha=1}^N f_\alpha \qquad a.e.\text{ in }\Omega\,.
\end{align*}
Then there exist a constant $c$, matrix fields $F_\alpha\in
L^{p_\alpha}(\Omega;\R^{n\times n})$ and a rotation $Q\in \SO(n)$ such
that 
\begin{equation*}
  D u = Q + \sum_{\alpha=1}^N F_\alpha \qquad a.e.\text{ in }\Omega\,,
\end{equation*}
and
\begin{align}\label{thmestimatesmanya}
  \|F_\alpha\|_{L^{p_\alpha}(\Omega;\R^{n\times n})}\le c
  \|f_\alpha\|_{L^{p_\alpha}(\Omega)}\,,\hskip3mm
\alpha=1,2,\dots N\,.
 \end{align}
The constant $c$ depends only on $n$, $p_1, p_2, \dots, p_N$ and $\Omega$ but
not on $u$ and $f_1,f_2,\dots, f_N$.
\end{theorem}
As above, we start from the case that $u$ is Lipschitz.
\begin{lemma}\label{lemmalipschitzcasemanyexponents}
Theorem~\ref{theo:rigiditypqnonlinearmanyexponents}   holds under  the additional assumption that
there exists a constant $M>n$ so that  $|D u|\le M$ a.e. The constant 
in the estimate~(\ref{thmestimatesmanya}) depends on $M$. 
\end{lemma}
\begin{proof}
  As above, we may assume $0\le f_\alpha\le 2M$, and proceed by induction on $N$.

{\em  Step 1: $p_N\le 2p_1$}. 
We first  reduce to the case 
  \begin{equation}\label{eqnormfalphaalpha1}
    \|f_\alpha\|_{p_\alpha}^{p_\alpha}\ge 
    \|f_{\alpha+1}\|_{p_{\alpha+1}}^{p_{\alpha+1}}\,,\hskip3mm
    \alpha=1, 2, \dots, N-1\,.
  \end{equation}
  If this does not hold for some $\alpha$, we set $\tilde f_{\alpha+1}=f_\alpha+f_{\alpha+1}$,
  estimate
$   \|\tilde f_{\alpha+1}\|_{p_{\alpha+1}}^{p_{\alpha+1}}\le
c \| f_{\alpha+1}\|_{p_{\alpha+1}}^{p_{\alpha+1}}+
c \| f_{\alpha}\|_{p_{\alpha}}^{p_{\alpha}}\le c \|
 f_{\alpha+1}\|_{p_{\alpha+1}}^{p_{\alpha+1}}$, and apply the result with the
 remaining $N-1$ exponents.

 By the rigidity in $L^{p_N}$ there is $Q\in \SO(n)$ with
\begin{align*}
    \|D u-Q\|_{p_N}^{p_N} \le c \|\dist(D u, \SO(n))\|_{p_N}^{p_N} \le c
    \sum_{\alpha=1}^N     \|f_\alpha \|_{p_N}^{p_N} \le c \|f_1\|_{p_1}^{p_1}\,,
 \end{align*}
where we used
 (\ref{eqnormfalphaalpha1}) and $\|f_\alpha\|_{p_N}^{p_N}\le
 (2M)^{p_N-p_\alpha}\|f_\alpha\|_{p_\alpha}^{p_\alpha}$. 
 We reduce to the case $Q=\Id$, expand and estimate 
\begin{align*}
    \left\| |D u -\Id|^2 \right\|_{p_1}^{p_1} = \int_\Omega |D u-\Id|^{2p_1} \dv{x}
    \le (2M)^{2p_1-p_N} \int_\Omega |D u-\Id|^{p_N}  \dv{x}\le c \|f_1\|_{p_1}^{p_1}\,.
\end{align*}
We set $\tilde f_1=f_1+|Du-\Id|^2$, $\tilde f_\alpha=f_\alpha$ for $\alpha\ge
2$. Since $|Eu-\Id|\le \sum_\alpha \tilde f_\alpha$ we can conclude using the
linear estimate.

{\em Step 2.} We define
\begin{equation*}
  \Lambda_k=\{(p_1, p_2,\dots,p_N)\in(1,\infty)^N: 
  p_\alpha\le p_{\alpha+1} \text{ for all }\alpha \text{ and }   
2^kp_1<p_N\leq 2^{k+1}p_1\}
\end{equation*}
and proceed by induction on $k$. Assume $(p_1, \dots, p_N)\in
\Lambda_{k+1}$. We define $q_1=2p_1$, $q_\alpha=\min\{2p_\alpha, p_N\}$ for
$2\le \alpha <N$, $q_N=p_N$. Then $(q_1, \dots, q_N)\in \Lambda_k$. 
Notice that, for all $\alpha$,  $q_\alpha\ge p_\alpha$ hence
$\|f_\alpha\|_{q_\alpha}^{q_\alpha}\le (2M)^{q_\alpha-p_\alpha}
\|f_\alpha\|_{p_\alpha}^{p_\alpha}$.  
We apply the estimate for the exponents $(q_1, \dots, q_N)$, and the same
functions $f_1,\dots, f_N$,
and obtain $Du=Q+\sum_\alpha F_\alpha$, with
\begin{equation*}
  \|F_\alpha\|_{q_\alpha}^{q_\alpha}\le c
  \|f_\alpha\|_{q_\alpha}^{q_\alpha}\le c \|f_\alpha\|_{p_\alpha}^{p_\alpha}. 
\end{equation*}
We reduce to the case $Q=\Id$ and estimate
\begin{equation*}
  |Eu-\Id|\le c \sum_{\alpha=1}^N f_\alpha + c \sum_{\alpha=1}^N |F_\alpha|^2\,.
\end{equation*}
We estimate, since $q_\alpha\le 2p_\alpha$ for all $\alpha$,
\begin{equation*}
  \| \, |F_\alpha|^2 \|_{p_\alpha}^{p_\alpha} =
  \|F_\alpha \|_{2p_\alpha}^{2p_\alpha} \le   
c  \|F_\alpha \|_{q_\alpha}^{q_\alpha} \le   
c  \|f_\alpha \|_{p_\alpha}^{p_\alpha}
\end{equation*}
and conclude with the linear estimate as above.
\end{proof}
\begin{proof}[Proof of Theorem \ref{theo:rigiditypqnonlinearmanyexponents}]
  As in the proof of Theorem \ref{theo:multikorn}, we may assume
  $\|f_\alpha\|_{p_\alpha}\le 
  \|f_{\alpha+1}\|_{p_{\alpha+1}}$.
  We choose $\lambda=2n$ and define $u_M$ as in the proof of Theorem
\ref{theo:rigiditypqnonlinear}. Let $\beta\in\{1, \dots,
  N\}$ be such that $\|f_\beta\|_{p_\beta}^{p_\beta} \ge
  \|f_\alpha\|_{p_\alpha}^{p_\alpha}$ for all $\alpha$. We define
  \begin{equation*}
f^M_\beta=f_\beta+2M\chi_{\Omega\setminus E} \text{ and  }
f^M_\alpha=f_\alpha \text{ for $\alpha\ne\beta$.}
  \end{equation*}
As in the proof of Theorem
\ref{theo:rigiditypqnonlinear}, we show that $\dist(Du_M, \SO(n))\le
\sum_\alpha f^M_\alpha$ and 
that 
$\|f^M_\beta\|_{p_\beta}\le c \|f_\beta\|_{p_\beta}$. We apply Lemma
\ref{lemmalipschitzcasemanyexponents} and conclude as 
in the proof of Theorem \ref{theo:rigiditypqnonlinear}.
\end{proof}


\appendix
\section{Extension}
The subsequent extension theorem for functions with mixed growth
follows immediately from the $L^2$-version in~\cite{Nitsche1981}. We include
a sketch of the proof for the convenience of the reader.

\begin{theorem}\label{theo:extension}
  Let $\varphi\in\Lip(\R^{n-1};\R)$ be a Lipschitz function with $\varphi(0)=0$ and 
Lipschitz constant $L$, let
$R>0$ and set  $\Omega=B(0,R)\cap\{(x',x_n)\in\R^{n-1}\times\R\colon x_n< \varphi(x')\}$.
Suppose that  $1<p<q<\infty$ and that $u\in W^{1,1}(\Omega;\R^n)$ with 
\begin{equation}\label{extdecomp}
  D u + D u^T=  f   + g\,, 
\end{equation}
where  $ f  \in L^p(\Omega;\MM{n}{n})$ and $g \in L^q(\Omega;\MM{n}{n})$.
Then there exists for $r=R/(2\sqrt{1+L^2})$ a function $w\in W^{1,1}(B(0,r);\R^n)$, 
and matrix fields $\fext$, $\gext$
such that
$w=u$, $\fext= f $, $\gext=g$ on $\Omega\cap B(0,r)$ and 
\begin{equation*}
  D w + D w^T= \fext + \gext \qquad \text{ on }B(0,r)
\end{equation*}
with
\begin{equation}\label{extest}
  \|\fext\|_{L^p(B(0,r))}\le c \| f  \|_{L^p(B(0,R)\cap\Omega)} \,, \hskip1cm
  \|\gext\|_{L^q(B(0,r))}\le c \|g \|_{L^q(B(0,R)\cap\Omega)} \,.
\end{equation}
The constant $c$ depends only on $n$, $p$, $q$, $\Omega$ but not
on $u$, $ f $, $g$.
\end{theorem}

\begin{proof}
  Let $\delta\in C^2(B(0,R)\setminus \Omega)$ be a function such that
  \begin{align*}
    2\dist(x,\Omega)\le \delta(x)\le C \dist(x,\Omega)
  \end{align*}
  and
  \begin{align}\label{distderivatives}
    |D^\alpha\delta(x)| \le C \delta^{1-|\alpha|}(x)\,,\quad \alpha\in\N^n\,,
  \end{align}
see, e.g.,~\cite{Stein1970}. Fix a function $\psi\in C^1(\R)$ with 
\begin{align}\label{psiprop}
 \int_1^2 \psi(\lambda)\dv{\lambda}=1\,,\quad
 \int_1^2 \lambda \psi(\lambda)\dv{\lambda}=0\,.
\end{align}
  We set $w=u$ on $\Omega$ and for $x\in B(0,r)\setminus\Omega$ we define
  \begin{align*}
    w(x)=\int_1^2 \psi(\lambda)\left[ u(x-\lambda \delta(x) e_n)
      - \lambda D\delta(x)u_n(x-\lambda \delta(x) e_n)\right] \dv{\lambda}\,.
  \end{align*}
For ease of notation we omit the arguments in the following calculations and write
$\delta=\delta(x)$ and $u=u(x-\lambda \delta(x) e_n)$ with the same convention for their derivatives.
By the chain rule
  \begin{align*}
    Dw(x)&\, =\int_1^2 \psi(\lambda)\left[ Du 
      (\Id - \lambda e_n\otimes D\delta )
      - \lambda D\delta \otimes Du_n (\Id - \lambda e_n\otimes D\delta )
      -\lambda  u_n D^2\delta\right] \dv{\lambda} \\
&\, = \int_1^2 \psi(\lambda)\left[ Du 
 - \lambda D_nu \otimes  D\delta 
      - \lambda D\delta \otimes
Du_n + \lambda^2 D_nu_n D\delta \otimes D\delta
      -\lambda u_nD^2\delta \right] \dv{\lambda}\,.
  \end{align*}
  Then the symmetric part of the gradient is given by 
  \begin{align*}
    Ew(x)=\int_1^2 \psi(\lambda)\left[ Eu 
 - \lambda (Eu e_n)\otimes  D\delta 
      - \lambda D\delta \otimes
(Eu e_n) + \lambda^2(Eu)_{nn} D\delta \otimes D\delta 
      -\lambda u_nD^2\delta \right] \dv{\lambda}\,.
  \end{align*}
In the last term we write
\begin{align*}
u_n(x-\lambda \delta(x) e_n)  =
u_n(x- \delta(x) e_n) +\int_1^\lambda
D_n u_n (x-s \delta(x) e_n)\delta(x) \dv{s}\,.
\end{align*}
In view of the second property in~\bref{psiprop} the weighted integral of
$u_n(x- \delta(x) e_n)$ is equal to zero, and the other term only depends on $(Eu)_{nn}$.
We recall~\bref{extdecomp} and define for $x\in B(0,r)\setminus\Omega$
  \begin{align*}
    \fext(x)&\, =\int_1^2 \psi(\lambda)\left[  f (x-\lambda \delta(x)e_n )
 - \lambda ( f (x-\lambda \delta(x)  e_n)e_n)\otimes  D\delta(x)  \right] \dv{\lambda} \\
&\, \qquad
      - \int_1^2 \psi(\lambda)\left[ \lambda D\delta(x) \otimes
( f (x-\lambda \delta(x)  e_n)e_n)  \right] \dv{\lambda} \\
&\, \qquad
      + \int_1^2 \psi(\lambda)\left[
\lambda^2 
 f _{nn}(x-\lambda \delta(x) e_n) D\delta(x) \otimes D\delta(x) 
    \right] \dv{\lambda} \\
&\, \qquad -
\int_1^2 \psi(\lambda)
\lambda \int_1^\lambda  f _{nn}(x-s \delta(x)e_n) \delta(x) \dv{s}\, D^2\delta(x)
\dv{\lambda}\,,
  \end{align*}
and use the analogous definition for $\gext$ in $x\in B(0,r)\setminus\Omega$.
On $B(0,r)\cap\Omega$ we set $\fext= f $ and $\gext=g$.
It remains to show that 
\begin{align*}
  \|\fext\|_{L^p(B(0,r)\setminus\Omega)}\le c \| f \|_{L^p(\Omega)}\,,\quad
\|\gext\|_{L^q(B(0,r)\setminus\Omega)}\le c \|g\|_{L^q(\Omega)}
\end{align*}
with a constant which only depends on  $n$,
$p$, $q$ and $\Omega$. The calculation is identical
to the proof of the estimate for the extension in~\cite[Lemma~4]{Nitsche1981} .
\end{proof}

\def\cprime{$'$}

\end{document}